\documentclass[12pt,reqno]{amsart}
\usepackage{amssymb,amscd,amsbsy,mathrsfs}
\newcommand{\lb}{\linebreak}

\renewcommand{\a}{\alpha}
\renewcommand{\b}{\beta}
\newcommand{\g}{\gamma}

\newcommand{\s}{\sigma}

\newcommand{\f}{\varphi}

\newcommand{\G}{\Gamma}

\newcommand{\h}{{\mathscr H}}

\newcommand{\R}{{\Bbb R}}
\newcommand{\Z}{{\Bbb Z}}

\newcommand{\0}{{\boldsymbol{0}}}

\newcommand{\bs}{\boldsymbol}

\newcommand{\bS}{{\boldsymbol S}}

\newcommand{\rf}[1]{(\ref{#1})}

\newcommand{\df}{\stackrel{\mathrm{def}}{=}}

\newcommand{\trace}{\operatorname{trace}}

\newcommand{\const}{\operatorname{const}}

\newcommand{\eeq}{\end{equation}}
\newcommand{\beq}{\begin{equation}}
\newcommand{\bay}{\begin{eqnarray}}
\newcommand{\ba}{\begin{align*}}
\newcommand{\ea}{\end{align*}}
\newcommand{\ey}{\end{eqnarray}}
\newcommand{\bey}{\begin{eqnarray*}}
\newcommand{\eey}{\end{eqnarray*}}

\newcommand{\be}{\infty}

\newcommand{\bl}{\blacksquare}

\newcommand{\Pf}{{\bf Proof. }}

\newtheorem{thm}{\hspace{\parindent}Theorem}[section]

\newtheorem{lem}[thm]{\hspace{\parindent}Lemma}

\usepackage{amsmath}

\makeatletter
\def\upintkern@{\mkern-7mu\mathchoice{\mkern-3.5mu}{}{}{}}
\def\upintdots@{\mathchoice{\mkern-4mu\@cdots\mkern-4mu}%
 {{\cdotp}\mkern1.5mu{\cdotp}\mkern1.5mu{\cdotp}}%
 {{\cdotp}\mkern1mu{\cdotp}\mkern1mu{\cdotp}}%
 {{\cdotp}\mkern1mu{\cdotp}\mkern1mu{\cdotp}}}

\newcommand{\UpMultiIntegral}[1]{%
  \edef\ints@c{\noexpand\upintop
    \ifnum#1=\z@\noexpand\upintdots@\else\noexpand\upintkern@\fi
    \ifnum#1>\tw@\noexpand\upintop\noexpand\upintkern@\fi
    \ifnum#1>\thr@@\noexpand\upintop\noexpand\upintkern@\fi
    \noexpand\upintop
    \noexpand\ilimits@
  }%
  \futurelet\@let@token\ints@a
}
\makeatother

\DeclareFontFamily{OMX}{mdbch}{}
\DeclareFontShape{OMX}{mdbch}{m}{n}{ <->s * [0.8]  mdbchr7v }{}
\DeclareFontShape{OMX}{mdbch}{b}{n}{ <->s * [0.8]  mdbchb7v }{}
\DeclareFontShape{OMX}{mdbch}{bx}{n}{<->ssub * mdbch/b/n}{}

\DeclareSymbolFont{uplargesymbols}{OMX}{mdbch}{m}{n}
\SetSymbolFont{uplargesymbols}{bold}{OMX}{mdbch}{b}{n}
\DeclareMathSymbol{\upintop}{\mathop}{uplargesymbols}{82}
\DeclareMathSymbol{\upointop}{\mathop}{uplargesymbols}{"48}

\DeclareFontEncoding{MDB}{}{}
\DeclareFontFamily{MDB}{mdbch}{}
\DeclareFontShape{MDB}{mdbch}{m}{n}{ <->s * [0.8]  mdbchrmb }{}
\DeclareFontShape{MDB}{mdbch}{b}{n}{ <->s * [0.8]  mdbchbmb }{}
\DeclareFontShape{MDB}{mdbch}{bx}{n}{<->ssub * mdbch/b/n}{}
\DeclareFontSubstitution{MDB}{cmr}{m}{n}
\DeclareSymbolFont{mathdesignB}{MDB}{mdbch}{m}{n}%
\SetSymbolFont{mathdesignB}{bold}{MDB}{mdbch}{b}{n}%
\DeclareMathSymbol{\upintclockwise}{\mathop}{mathdesignB}{128}
\DeclareMathSymbol{\upointclockwise}{\mathop}{mathdesignB}{130}
\DeclareMathSymbol{\upointctrclockwise}{\mathop}{mathdesignB}{132}
\DeclareMathSymbol{\upoiint}{\mathop}{mathdesignB}{134}
\DeclareMathSymbol{\upoiiint}{\mathop}{mathdesignB}{136}

\makeatletter
\newcommand{\upint}{\DOTSI\upintop\ilimits@}
\newcommand{\upoint}{\DOTSI\upointop\ilimits@}
\makeatother

\theoremstyle{remark}

\newtheorem*{rem*}{Remark}

\newcommand\fM{\frak M}

\newcommand\dg{\frak D}

\newcommand\mB{\mathcal{B}}

\newcommand\rd{{\rm d}}

\begin{document}

\numberwithin{equation}{section}

\numberwithin{equation}{section}

\title{Commutator estimates for functions of noncommuting self-adjoint operators}
\author{V.V. Peller}
\thanks{The research on \S\:2-3  is supported by 
Russian Science Foundation [grant number 23-11-00153].
The research on \S\:4 is supported by a grant of the Government of the Russian Federation for the state support of scientific research, carried out under the supervision of leading scientists, agreement  075-15-2025-013.\newline
}


\

%
\maketitle

\begin{abstract}
We study properties of the calculus $\f\mapsto\f(A,B)$ for self-adjoint operators with commutator
$[A,B]$ in the Schaten--von Neumann class $\bS_p$. It turns out that this calculus defined on the Besov class $B_{\be,1}^1(\R^2)$ in the case $p\le2$ admits a commutator Lipschitz estimate and is multiplicative modulo $\bS_p$. On the other hand in the case $p>2$ there is no commutator Lipschitz estimate. There is no commutator Lipschitz estimate in the operator norm as well.
\end{abstract}

\section{\bf Introduction}
\setcounter{equation}{0}
\label{In}

\

In this paper we continue studying properties of the calculus $\f\mapsto \f(A,B)$ for not necessarily commuting self-adjoint operators $A$ and $B$. This study was initiated in \cite{ANP}. For a function $\f$ on $\R^2$, the function
$\f(A,B)$ of self-adjoint operators $A$ and $B$ is defined as the double operator integral
$$
\f(A,B)=\iint_{\R^2}\f(x,y)\,\rd E_A(x)\,\rd E_B(y),
$$
where $E_A$ and $E_B$ are the spectral measures of $A$ and $B$.

Recall that the theory of double operator integrals of the form
\bay
\label{dvoinoi}
\iint \Phi(x,y)\,\rd E_1(x)Q\,\rd E_2(y)
\ey
was developed by M.Sh.~Birman and M.Z.~Solomyak in \cite{BS1}--\cite{BS3}. Here $E_1$ and $E_2$ are spectral measures on Hilbert space, $\Phi$ is a bounded measurable function and $Q$ is a bounded linear operator on Hilbert space. We also refer the reader to \cite{AP1}, \cite{Pe3} for more information on double operator integrals.

The maximal class of functions $\Phi$ for which the double operator integral \rf{dvoinoi} can be defined as a bounded linear operator for an arbitrary bounded operator $Q$ is called the class of {\it Schur multipliers} with respect to $E_1$ and $E_2$. This class is denoted by $\fM_{E_1,E_2}$. 

It is well known (see \cite{Pe1}, \cite{AP1} and \cite{AP4}) that $\Phi\in\fM_{E_1,E_2}$ if and only if $\Phi$ belongs to the {\it Haagerup tensor product} $L^\be_{E_1}\otimes_{\rm h}L^\be_{E_1}$
of the spaces $L^\be_{E_1}$ and $L^\be_{E_2}$, i.e. $\Phi$ admits a representation
$$
\Phi(x,y)=\sum_nu_n(x)v_n(y),
$$
with functions $u_n$ in $L^\be_{E_1}$ and $v_n$ in $L^\be_{E_2}$ satisfying the condition
$$
\sum_n|u_n|^2\in L^\be_{E_1}\quad\mbox{and}\quad\sum_n|v_n|^2\in L^\be_{E_2}.
$$

It was observed in \cite{ANP} that if $\f$ is a function on $\R^2$ that belongs to the homogeneous Besov class $B_{\be,1}^1(\R^2)$, then $\f\in L^\be_{E_A}\otimes_{\rm h}L^\be_{E_B}$ for arbitrary bounded self-adjoint operators $A$ and $B$ which allows one to define the functional calculus
$$
\f\mapsto\f(A,B),\quad \f\in B_{\be,1}^1(\R^2).
$$
We refer the reader to \cite{AP1} and \cite{Pe4} for the definition and basic properties of Besov spaces.

It was shown in \cite{ANP} that this functional calculus satisfies a Lipschitz type estimate in the Schatten--von Neumann norm of $\bS_p$ for $p\in[1,2]$, i.e. the following inequality holds
\bay
\label{LipSp}
\!\|\f(A_1,B_1)-\f(A_2,B_2)\|_{\bS_p}
\le\const\|\f\|_{B_{\be,1}^1}\!\max\{\|A_2-A_1\|_{\bS_p},\|B_2-B_1\|_{\bS_p}\}
\ey
whenever $A_1$, $B_1$, $A_2$ and $B_2$ are bounded self-adjoint operators such that
$A_2-A_1\in\bS_p$ and $B_2-B_1\in\bS_p$. We refer the reader to \cite{GK} for the definition and basic properties of Schatten--von Neumann classes.

On the other hand, it was proved in \cite{ANP} that inequality \rf{LipSp} is false in the norm of 
$\bS_p$ for $p>2$ and it is also false in the operator norm.

In \cite{AP2} the authors considered the functional calculus $\f\mapsto\f(A,B)$
for pairs of almost commuting bounded self-adjoint operators $A$ and $B$. Recall that $A$ and $B$ are called {\it almost commuting} if the commutator $[A,B]\df AB-BA$ belongs to the trace class $\bS_1$. It was established in \cite{AP2} that for almost commuting bounded self-adjoint operators $A$ and $B$, this calculus is almost multiplicative on the
the homogeneous Besov class $B_{\be,1}^1(\R^2)$, i.e.
$$
\f(A,B)\psi(A,B)-\psi(A,B)\f(A,B)\in\bS_1
$$
whenever $\f$ and $\psi$ belong to $B_{\be,1}^1(\R^2)$. Moreover, in \cite{AP2} the Helton--Howe trace formula (see \cite{HH}) was extended to the homogeneous Besov class 
$B_{\be,1}^1(\R^2)$:
\begin{align*}
\label{exHH}
\trace\big({\rm i}\big(\f(A,B)&\psi(A,B)-\psi(A,B)\f(A,B)\big)\big)\nonumber\\[.2cm]
=&\frac{1}{2\pi}
\iint_{\R^2}\left(\frac{\partial\f}{\partial x}\frac{\partial\psi}{\partial y}-
\frac{\partial\f}{\partial y}\frac{\partial\psi}{\partial x}\right)g(x,y)\,dx\,dy,\quad\f,\,\psi\in
B_{\be,1}^1(\R^2),
\end{align*}
where $g$ is the {\it Pincus principal function} introduced in \cite{Pi} that is associated with the almost commuting self-adjoint operators $A$ and $B$.

In this paper we consider the functional calculus $\f\mapsto\f(A,B)$, $\f\in B_{\be,1}^1(\R^2)$ for bounded self-adjoint operators $A$ and $B$ with commutators in the Schatten--von Neumann class $\bS_p$. 

\

\section{\bf Triple operator integrals}
\setcounter{equation}{0}
\label{triple}

\

Both in \cite{ANP} and \cite{AP2} an important role was played by {\it triple operator integrals}, i.e. 
expressions of the form
\bay
\label{tri}
\iiint \Phi(x,y,z)\,\rd E_1(x)Q_1\,\rd E_2(y)Q_2\,\rd E_3(z).
\ey
Here $E_1$, $E_2$ and $E_3$ are spectral measures on Hilbert space, $Q_1$ and $Q_2$ are bounded linear operators on Hilbert space and $\Phi$ is a bounded measurable function. We call $\Phi$ the {\it symbol} of the triple operator integral \rf{tri}.

As in the case of double operator integrals, for such triple operator integrals to make sense the function $\Phi$ and the operators $Q_1$ and $Q_2$ have to satisfy certain conditions.

We mention here several approaches to triple operator integrals.

In \cite{Pe2} triple operator integrals of the form \rf{tri} were defined for functions $\Phi$ in the integral projective tensor product 
$L^\be_{E_1}\hat\otimes_{\rm i}L^\be_{E_2}\hat\otimes_{\rm i}L^\be_{E_3}$, see \cite{Pe2}
for the definition. It turns out that if
$\Phi\in L^\be_{E_1}\hat\otimes_{\rm i}L^\be_{E_2}\hat\otimes_{\rm i}L^\be_{E_3}$, $Q_1$ and $Q_2$ are bounded operators, then the triple operator integral in \rf{tri} determines a bounded linear operator. Also, for $\Phi\in L^\be_{E_1}\hat\otimes_{\rm i}L^\be_{E_2}\hat\otimes_{\rm i}L^\be_{E_3}$, the conditions that $Q_1\in\bS_p$ and $Q_2\in\bS_q$ with $1/p+1/q\le1$ ensure that the triple operator integral in \rf{tri} belongs to the Schatten--von Neumann class $\bS_r$, where $1/r=1/p+1/q$.

We dwell on the next two approaches in more detail. In \cite{JTT} the authors considered triple operator integrals of the form \rf{tri} in the case when $\Phi$ belongs to the {\it Haagerup tensor product} $L^\be_{E_1}\otimes_{\rm h}L^\be_{E_2}\otimes_{\rm h}L^\be_{E_3}$ of the $L^\infty$ spaces which consists of functions $\Phi$ of the form
$$
\Phi(x,y,z)=\sum_{j,k\ge0}\a_j(x_1)\b_{jk}(x_2)\g_k(x_3)
$$
where $\{\a_j\}_{j\ge0},~\{\g_k\}_{k\ge0}\in L^\be(\ell^2)$, 
$\{\b_{jk}\}_{j,k\ge0}\in L^\be(\mB)$, where $\mB$ is the {\it space of matrices that induce bounded linear operators on} $\ell^2$ equipped with the operator norm. For such functions $\Phi$ the triple operator integral in \rf{tri} can be expressed in the following way:
\begin{multline*}
\iiint \Phi(x,y,z)\,\rd E_1(x)Q_1\,\rd E_2(y)Q_2\,\rd E_3(z)\\[.2cm]
=\sum_{j,k\ge0}\left(\int\a_j\,\rd E_1\right)Q_1\left(\int\b_{jk}\,\rd E_2\right)
Q_2\left(\int\g_k\,\rd E_3\right)\nonumber\\[.2cm]
=\lim_{M,N\to\be}~\sum_{j=0}^N\sum_{k=0}^M
\left(\int\a_j\,\rd E_1\right)Q_1\left(\int\b_{jk}\,\rd E_2\right)
Q_2\left(\int\g_k\,\rd E_3\right),
\end{multline*}
where the limit on the right exists in the weak operator topology. Also, the triple operator integral
in \rf{tri} can be expressed as the product of the block operator matrices
\begin{multline*}
\iiint \Phi(x,y,z)\,\rd E_1(x)Q_1\,\rd E_2(y)Q_2\,\rd E_3(z)\\[.2cm]
=\left(\begin{matrix}
A_1&A_2&A_3&\cdots
\end{matrix}\right)
\left(\begin{matrix}
B_{11}&B_{12}&B_{13}&\cdots\\[.2cm]
B_{21}&B_{22}&B_{23}&\cdots\\[.2cm]
B_{31}&B_{32}&B_{33}&\cdots\\[.2cm]
\vdots&\vdots&\vdots&\ddots
\end{matrix}\right)
\left(\begin{matrix}
\G_1\\[.2cm]
\G_2\\[.2cm]
\G_3\\[.2cm]
\vdots
\end{matrix}\right),
\end{multline*}
where 
$$
A_j\df\left(\int\a_j\,\rd E_1\right)Q_1,\quad B_{jk}\df\int\b_{jk}\,\rd E_2
\quad\mbox{and}\quad\G_k\df Q_2\int\g_{k}\,\rd E_3.
$$
All three block operator matrices on the right-hand side determine bounded linear operators, see \cite{AP3}.

It was established in \cite{AP3} that if $\Phi\in L^\be_{E_1}\otimes_{\rm h}L^\be_{E_2}\otimes_{\rm h}L^\be_{E_3}$, $Q_1\in\bS_p$, $Q_2\in\bS_q$ with $p,q\in[2,\be]$, then the triple operator integral in \rf{tri} belongs to $\bS_r$, where $1/r=1/p+1/q$. Here and in what follows for convenience, by $T\in\bS_\be$ we mean that $T$ is a bounded linear operator.

We proceed now to the most important approach for us to triple operator integrals. This approach was offered in \cite{ANP} and was slightly refined in \cite{AP3}. It allows us to define triple operator integrals of the form \rf{tri} for functions $\Phi$ in Haagerup-like tensor products of $L^\be$ spaces.

We say that a measurable function $\Phi$ of three variables {\it belongs to the Haagerup-like tensor product 
$L^\be(E_1)\!\otimes_{\rm h}\!L^\be(E_2)\!\otimes^{\rm h}\!L^\be(E_3)$ of the first kind} if it admits a representation
\bay
\label{yaH}
\Phi(x,y,z)=\sum_{j,k\ge0}\a_j(x)\b_{k}(y)\g_{jk}(z)
\ey
with $\{\a_j\}_{j\ge0},~\{\b_k\}_{k\ge0}\in L^\be(\ell^2)$ and 
$\{\g_{jk}\}_{j,k\ge0}\in L^\be(\mB)$. We equip the space \lb$L^\be(E_1)\!\otimes_{\rm h}\!L^\be(E_2)\!\otimes^{\rm h}\!L^\be(E_3)$ with the norm
$$
\|\Phi\|_{L^\be\otimes_{\rm h}\!L^\be\otimes^{\rm h}\!L^\be}
\df\inf\big\|\{\a_j\}_{j\ge0}\big\|_{L^\be(\ell^2)}
\big\|\{\b_k\}_{k\ge0}\big\|_{L^\be(\ell^2)}
\big\|\{\g_{jk}\}_{j,k\ge0}\big\|_{L^\be(\mB)},
$$
the infimum being taken over all representations of the form \rf{yaH}.

Let us now define {\it triple operator integrals of the first gender} with integrand $\Phi$ in the tensor product
$L^\be(E_1)\!\otimes_{\rm h}\!L^\be(E_2)\!\otimes^{\rm h}\!L^\be(E_3)$.

Suppose that $1\le p\le2$. For 
$\Phi\in L^\be(E_1)\!\otimes_{\rm h}\!L^\be(E_2)\!\otimes^{\rm h}\!L^\be(E_3)$, for a bounded linear operator $T$, and for an operator $R$ of class $\bS_p$, we define the triple operator integral
$$
W=\iint\!\!\upint\Phi(x,y,z)\,dE_1(x)T\,dE_2(y)R\,dE_3(z)
$$
{\it of the first gender}
as the following continuous linear functional on $\bS_{p'}$,
$1/p+1/p'=1$ (on the class of compact operators in the case $p=1$):
$$
Q\mapsto
\trace\left(\left(
\iiint
\Phi(x,y,z)\,dE_2(y)R\,dE_3(z)Q\,dE_1(x)
\right)T\right).
$$
It follows easily from the above Schatten--von Neumann properties of triple operator integrals with symbols in the Haagerup tensor product $L^\be\otimes_{\rm h}L^\be\otimes_{\rm h}L^\be$
that $W\in\bS_p$ and
\bay
\label{perrod}
\|W\|_{\bS_p}\le\|\Psi\|_{L^\be\otimes_{\rm h}\!L^\be\otimes^{\rm h}\!L^\be}
\|T\|_{\bS_p}\|R\|,\quad1\le p\le2,
\ey
see \cite{ANP} and \cite{AP4}. Note that a considerably more general result on Schatten--von Neumann properties of triple operator integrals of the first gender is given in Theorem 5.1 of \cite{AP4}.

Let us proceed now to Haagerup tensor products of the second kind and triple operator integrals of the second gender.

We say that a function $\Phi$ belongs to the {\it Haagerup-like tensor product \lb$L^\be(E_1)\!\otimes^{\rm h}\!L^\be(E_2)\!\otimes_{\rm h}\!L^\be(E_3)$
of the second kind} if
$\Phi$ admits a representation
\bay
\label{preds}
\Phi(x,y,z)=\sum_{j,k\ge0}\a_{jk}(x)\b_{j}(y)\g_k(z)
\ey
where $\{\b_j\}_{j\ge0},~\{\g_k\}_{k\ge0}\in L^\be(\ell^2)$, 
$\{\a_{jk}\}_{j,k\ge0}\in L^\be(\mB)$. The norm of $\Phi$ in 
the space $L^\be\otimes^{\rm h}\!L^\be\otimes_{\rm h}\!L^\be$ is defined by
$$
\|\Phi\|_{L^\be\otimes^{\rm h}\!L^\be\otimes_{\rm h}\!L^\be}
\df\inf\big\|\{\a_{jk}\}_{j,k\ge0}\big\|_{L^\be(\mB)}
\big\|\{\b_j\}_{j\ge0}\big\|_{L^\be(\ell^2)}
\big\|\{\g_{k}\}_{k\ge0}\big\|_{L^\be(\ell^2)},
$$
the infimum being taken over all representations of the form \rf{preds}.

We proceed now to the definition of {\it triple operator integrals of the second gender} with integrand in the tensor product
$L^\be(E_1)\!\otimes^{\rm h}\!L^\be(E_2)\!\otimes_{\rm h}\!L^\be(E_3)$.

Suppose that 
$\Phi\in L^\be(E_1)\!\otimes^{\rm h}\!L^\be(E_2)\!\otimes_{\rm h}\!L^\be(E_3)$,
$T$ is a bounded linear operator, and $R\in\bS_p$, $1\le p\le2$. The continuous linear functional 
$$
Q\mapsto
\trace\left(\left(
\iiint\Phi(x,y,z)\,dE_3(z)Q\,dE_1(x)T\,dE_2(y)
\right)R\right)
$$
on the class $\bS_{p'}$ (on the class of compact operators in the case $p=1$) 
determines an operator $W$ of class $\bS_p$, which
we call the triple operator integral
$$
W=\upint\!\!\!\iint\Phi(x,y,z)\,dE_1(x)T\,dE_2(y)R\,dE_3(z)
$$
{\it of the second gender}.

It is easy to see that
\bay
\label{vtorod}
\|W\|_{\bS_p}\le
\|\Psi\|_{L^\be\otimes^{\rm h}\!L^\be\otimes_{\rm h}\!L^\be}
\|T\|\cdot\|R\|_{\bS_p},
\ey
see \cite{ANP} and \cite{AP4}.

A considerably more general result on Schatten--von Neumann properties of triple operator integrals of the second gender is given in Theorem 5.2 of \cite{AP4}.

\

\section{\bf Commutator Lipschitz estimates in $\bS_p$ in the case $\bs{p\in[1,2]}$}
\setcounter{equation}{0}
\label{Gosha}

\

For a differentiable function $\f$ on $\R^2$, we consider the divided differences
$\dg^{[1]}\f$ and $\dg^{[2]}\f$ on $\R^3$ by
$$
\big(\dg^{[1]}\f)(x_1,x_2,y\big)\df
\left\{\begin{array}{ll}
\frac{\f(x_1,y)-\f(x_2,y)}{x_1-x_2},&x_1\ne x_2,\\[.2cm]
\frac{\partial\f}{\partial x}(x_1,y),&x_1=x_2.
\end{array}
\right.
$$
and
$$
\big(\dg^{[2]}\f\big)(x,y_1,y_2)\df
\left\{\begin{array}{ll}
\frac{\f(x,y_1)-\f(x,y_2)}{y_1-y_2},&y_1\ne y_2,\\[.2cm]
\frac{\partial\f}{\partial y}(x_1,y),&y_1=y_2.
\end{array}
\right.
$$

It was shown in \cite{ANP} that if $\f$ is a bounded function on $\R^2$ whose Fourier transform is supported 
in the ball $\{\xi\in\R^2:~\|\xi\|\le1\}$, then
$$
\big(\dg^{[1]}\f)(x_1,x_2,y\big)=
\sum_{j,k\in\Z}\frac{\sin(x_1-j\pi)}{x_1-j\pi}\cdot\frac{\sin(x_2-k\pi)}{x_2-k\pi}
\cdot\frac{\f(j\pi,y)-\f(k\pi,y)}{j\pi-k\pi}.
$$
For $j=k$, it is assumed that 
$$
\frac{\f(j\pi,y)-\f(k\pi,y)}{j\pi-k\pi}
=\frac{\partial\f(x,y)}{\partial x}\Big|_{(j\pi,y)}.
$$
Moreover,
$$
\sup_{y\in\R}\left\|\left\{\frac{\f(j\pi,y)-\f(k\pi,y)}{j\pi-k\pi}
\right\}_{j,k\in\Z}\right\|_\mB\le\const\|\f\|_{L^\be(\R^2)},
$$
see \cite{ANP}.

Since it is well known that
$$
\sum_{j\in\Z}\frac{\sin^2(x_1-j\pi)}{(x_1-j\pi)^2}
=\sum_{k\in\Z}\frac{\sin^2(x_2-k\pi)}{(x_2-k\pi)^2}=1,
\quad x_1~x_2\in\R,
$$
it follows that for such functions $\f$,
$$
\dg^{[1]}\f\in L^\be\otimes_{\rm h}\!L^\be\otimes^{\rm h}\!L^\be\quad
\mbox{and}\quad
\|\dg^{[1]}\f\|_{L^\be\otimes_{\rm h}\!L^\be\otimes^{\rm h}\!L^\be}
\le\const\|\f\|_{L^\be}.
$$
By rescaling, one can deduce from this that for bounded functions $\f$ on
$\R^2$ whose Fourier transform is supported in $\{\xi\in\R^2:~\|\xi\|\le\s\}$, we have
$$
\dg^{[1]}f\in L^\be\otimes_{\rm h}\!L^\be\otimes^{\rm h}\!L^\be\quad
\mbox{and}\quad
\|\dg^{[1]}\f\|_{L^\be\otimes_{\rm h}\!L^\be\otimes^{\rm h}\!L^\be}
\le\const\s\|\f\|_{L^\be},
$$
see \cite{ANP}.

It follows easily that if $\f$ is a function in the Besov class $B_{\be,1}^1(\R^2)$, then
$$
\dg^{[1]}\f\in L^\be\otimes_{\rm h}\!L^\be\otimes^{\rm h}\!L^\be\quad
\mbox{and}\quad
\big\|\dg^{[1]}\f\big\|_{L^\be\otimes_{\rm h}\!L^\be\otimes^{\rm h}\!L^\be}
\le\const\|\f\|_{B_{\be,1}^1}.
$$

Similarly, for $\f\in B_{\be,1}^1(\R^2)$,
$$
\dg^{[2]}\f\in L^\be\otimes^{\rm h}\!L^\be\otimes_{\rm h}\!L^\be
\quad\mbox{and}\quad
\big\|\dg^{[2]}\f\big\|_{L^\be\otimes^{\rm h}\!L^\be\otimes_{\rm h}\!L^\be}
\le\const\|\f\|_{B_{\be,1}^1},
$$
see \cite{ANP}.

The following theorem says that our calculus $\f\mapsto\f(A,B)$ admits a commutator Lipschitz
estimate in $\bS_p$ with $p\in[1.2]$; it generalizes Theorem 4.1 of \cite{AP2} from the case $p=1$.

\begin{thm}
\label{komutp}
Let $1\le p\le2$. Suppose that
$A$ and $B$ are self-adjoint operators and $Q$ is a bounded linear operator such that $[A,Q]\in\bS_p$ and $[B,Q]\in\bS_p$. 
Let $\f\in B_{\be,1}^1(\R^2)$.
Then $[\f(A,B),Q\big]\in\bS_p$,
\begin{align*}
\big[\f(A,B),Q\big]&=
\upint\!\!\!\iint\frac{\f(x,y_1)-\f(x,y_2)}{y_1-y_2}\,dE_A(x)\,dE_B(y_1)[B,Q]\,dE_B(y_2)\nonumber
\\[.2cm]
&+
\iint\!\!\upint\frac{\f(x_1,y)-\f(x_2,y)}{x_1-x_2}\,dE_A(x_1)[A,Q]\,dE_A(x_2)\,dE_B(y)
\end{align*}
and
$$
\big\|[\f(A,B),Q\big]\big\|_{\bS_p}
\le\const\|\f\|_{B_{\be,1}^1(\R^2)}\big(\big\|[A,Q]\big\|_{\bS_p}+
\big\|[B,Q]\big\|_{\bS_p}\big).
$$
\end{thm}

The proof of Theorem \ref{komutp} can be obtained by following the line of the proof of Theorem 4.1 of \cite{AP2} and applying inequalities \rf{perrod} and \rf{vtorod}.

The following theorem is a generalization of Theorem 4.2 of \cite{AP2}. It shows that in the case 
$[A,B]\in\bS_p$, $1\le p\le2$, for self-adjoint operators $A$ and $B$ our calculus
$\f\mapsto\f(A,B)$, $f\in B_{\be,1}^1(\R^2)$ is multiplicative modulo $\bS_p$.

\begin{thm}
\label{glavp}
Suppose that $1\le p\le2$.
Let $A$ and $B$ be self-adjoint operators such that $[A,B]\in\bS_p$ and let $\f$ and $\psi$ be functions in the Besov class $B_{\be,1}^1(\R^2)$. Then
\begin{align*}
\big[\f(A,B),\psi(A,B)\big]\!&=\!
\upint\!\!\!\iint
\frac{\f(x,y_1)-\f(x,y_2)}{y_1-y_2}\,dE_A(x)\,dE_B(y_1)[B,\psi(A,B)]\,dE_B(y_2)
\\[.2cm]
&+\!\!
\iint\!\!\upint\!
\frac{\f(x_1,y)-\f(x_2,y)}{x_1-x_2}dE_A(x_1)[A,\psi(A,B)]dE_A(x_2)dE_B(y)
\end{align*}
and
$$
\label{ner}
\big\|[\f(A,B),\psi(A,B)\big]\big\|_{\bS_p}
\le\const\|\f\|_{B_{\be,1}^1(\R^2)}\|\psi\|_{B_{\be,1}^1(\R^2)}
\big\|[A,B]\big\|_{\bS_p}.
$$
\end{thm}

The proof of Theorem \ref{glavp} can be obtained by following the line of the proof of Theorem 4.2 of \cite{AP2} and applying inequalities \rf{perrod} and \rf{vtorod}.

\

\section{\bf No commutator Lipschitz estimates in $\bS_p$ in the case $\bs{p>2}$}
\setcounter{equation}{0}

\

The main results of this section show that the result of 
\S\;\ref{Gosha} 
do 
not 
generalize 
to the case $p>2$.

\begin{thm}
\label{Sp}
Suppose that $p>2$. Then there is no positive number $K$ such that 
$$
 \big\|[\f(A,B),R]\big\|_{\bS_p}\le K\|\f\|_{B_{\be,1}^1(\R^2)}\max\{\|[A,R]\|_{\bS_p},\|[B,R]\|_{\bS_p}\}
 $$
 whenever $\f\in B_{\be,1}^1(\R^2)$, 
$R$ is a bounded linear operator, and $A$ and $B$ are finite rank self-adjoint operators on Hilbert space.
\end{thm}

The same is true in the operator norm

\begin{thm}
\label{opno}
There is no positive number $K$ such that 
$$
 \big\|[\f(A,B),R]\big\|\le K\|\f\|_{B_{\be,1}^1(\R^2)}\max\{\|[A,R]\|,\|[B,R]\|\}
 $$
 whenever $\f\in B_{\be,1}^1(\R^2)$, 
$R$ is a bounded linear operator, and $A$ and $B$ are finite rank self-adjoint operators on Hilbert space.
\end{thm}


 \begin{lem}
 \label{sved}
 Let $1\le p\le\be$. Suppose that there exists a positive number $K$ such that
 \bay
 \label{kuzya}
 \big\|[\f(A,B),R]\big\|_{\bS_p}\le K\|\f\|_{B_{\be,1}^1(\R^2)}\max\{\|[A,R]\|_{\bS_p},\|[B,R]\|_{\bS_p}\}
 \ey
whenever $\f\in B_{\be,1}^1(\R^2)$, 
$R$ is a bounded linear operator, and $A$ and $B$ are finite rank self-adjoint operators on a Hilbert space $\h$. 

Then
\bay
\label{zyuzya}
\|\f(A_1,B_1)-\f(A_2,B_2)\|_{\bS_p}\!\le 
K\|\f\|_{B_{\be,1}^1(\R^2)}\max\{\|A_1-B_1\|_{\bS_p},\|A_2-B_2\|_{\bS_p}\}
\ey
whenever $A_1$, $A_2$, $B_1$ and $B_2$ are finite rank self-adjoint operators.
 \end{lem}
 
 Recall that in the case $p=\be$, by $\bS_p$ we understand the space of bounded linear operators with the operator norm.
 
 \Pf For self-adjoint operators $A_1$, $A_2$, $B_1$ and $B_2$ on $\h$, we consider the self-adjoint operators 
 $A$ and $B$ on $\h\oplus\h$ defined by 
 $$
 A=\left(\begin{matrix}A_1&\0\\[.2cm]\0&A_2\end{matrix}\right)\quad\mbox{and}\quad
 B=\left(\begin{matrix}B_1&\0\\[.2cm]\0&B_2\end{matrix}\right).
 $$
 We define now the bounded operator $R$ on $\h\oplus\h$ by
 $$
 R=\left(\begin{matrix}\0&I\\[.2cm]\0&\0\end{matrix}\right).
 $$
 Clearly,
 $$
 \f(A,B)R-R\f(A,B)=\left(\begin{matrix}\0&\f(A_1,B_1)-\f(A_2,B_2)\\[.2cm]\0&\0\end{matrix}\right).
 $$
 On the other hand,
 $$
 AR-RA=\left(\begin{matrix}\0&A_1-A_2\\[.2cm]\0&\0\end{matrix}\right)\quad\mbox{and}\quad
 BR-RB=\left(\begin{matrix}\0&B_1-B_2\\[.2cm]\0&\0\end{matrix}\right).
 $$
 Clearly, inequality \rf{kuzya} implies inequality \rf{zyuzya}. $\bl$
 
 \medskip
  
 {\bf Proof of Theorems \ref{Sp} and \ref{opno}.} It is easy to see that Theorems \ref{Sp} and \ref{opno} follow immediately from Lemma \ref{sved} and the following result from \S\:8 of \cite{ANP}:
 
 {\it Let $2<p\le\be$.
Then there is no positive number $K$ such that 
$$
\|\f(A_1,B)-\f(A_2,B)\|_{\bS_p}\le K\|\f\|_{B^1_{{\be,1}}}\|A_1-A_2\|_{\bS_p}
$$ 
for all functions $\f$ on in the Besov class $B_{\be,1}^1(\R^2)$ and for all finite rank self-adjoint operators $A_1,\,A_2,\,B$.} $\bl$

\
 
\begin{footnotesize}
 
\noindent
\begin{tabular}{p{8cm}p{15cm}}
St.Petersburg State University  \\
Universitetskaya nab., 7/9  \\
199034 St.Petersburg, Russia  \\
\\

St.Petersburg Department \\
Steklov Institute of Mathematics   \\
Russian Academy of Sciences   \\
Fontanka 27, 191023 St.Petersburg,
Russia\\
email: email: peller@math.msu.edu
\end{tabular}
\end{footnotesize}


\begin{thebibliography}{99}

\bibitem[ANP]{ANP} {\sc A.B. Aleksandrov, F.L. Nazarov} and
{\sc V.V. Peller}, {\em Functions of noncommuting self-adjoint operators under perturbation and estimates of triple operator integrals}, Advances in Math.
{\bf295} (2016), 1--52.


\bibitem[AP1]{AP1}{\sc Aleksandrov A. B. , Peller V. V.}, {\it Operator Lipschitz functions}, Uspekhi Matem. Nauk. {\bf71:4} (2016), 3--106 (Russian).
English transl.: Russian Math. Surveys, {\bf71:4} (2016), 605--702.


\bibitem[AP2]{AP2}  {\sc A.B. Aleksandrov} and {\sc V.V. Peller},
{\em Functions of almost commuting operators and an extension of the Helton--Howe trace formula}, J. Funct. Anal. {\bf271} (2016), 3300--3322.


\bibitem[AP3]{AP3}  {\sc A.B. Aleksandrov} and {\sc V.V. Peller},
{\it Multiple operator integrals, Haagerup and Haagerup-like tensor products, and operator ideals},
Bull. London Math. Soc. {\bf49} (2017), 463--479.

\bibitem[AP4]{AP4} {\sc A.B. Aleksandrov} and {\sc V.V. Peller}, {\it Haagerup tensor products and Schur multipliers}, Algebra i Analiz {\bf36:5} (2024), 70--85.


\bibitem[BS1]{BS1} {\sc Birman M. S., Solomyak M. Z.}, 
{\em Double Stieltjes operator integrals}, 
Problems of Math. Phys.,
Leningrad. Univ. 1 (1966), 33--67 (Russian). English transl., Topics Math. Physics 1, 25--54 . Consultants
Bureau Plenum Publishing Corporation, New York, 1967.
\bibitem[BS2]{BS2} {\sc Birman M. S., Solomyak M. Z.}, 
{\em Double Stieltjes operator integrals. II}, Problems of Math. Phys.,
Leningrad. Univ. 2 (1967), 26--60 (Russian). English transl., Topics Math. Physics 2, 19--46. Consultants
Bureau Plenum Publishing Corporation, New York, 1968.
\bibitem[BS3]{BS3} {\sc Birman M. S., Solomyak M. Z.}, 
{\em Double Stieltjes operator integrals. III},
Problems of Math. Phys. {\bf6}, Leningrad. Univ, 27--53 (1973) (Russian).
\bibitem[GK]{GK} I.C. Gohberg and M.G. Krein, {\it Introduction to the theory
of linear nonselfadjoint operators in Hilbert space,} Nauka, Moscow, 1965.
\bibitem[HH]{HH} {\sc J.W. Helton} and {\sc R. Howe},
{\em Integral operators, commutators, traces, index, and homology}, 
in ``Lecture Notes in Math.'', vol. {\bf345}, pp. 141--209,
Springer-Verlag, New York, 1973.
\bibitem[JTT]{JTT} {\sc K. Juschenko, I.G. Todorov} and {\sc L. Turowska}, {\em Multidimensional operator multipliers}, Trans. Amer. Math. Soc. {\bf361}
(2009), 4683--4720.
\bibitem[Pe1]{Pe1} {\sc Peller V. V.}, 
{\em Hankel operators in the perturbation theory of unitary and self-adjoint
operators}, Funk, anal. i ego pril. {\bf19}:2  (1985),  37--51 (in Russian); English transl.,
Funct. Anal. Appl., {\bf19}:2 (1985), 111--123.

\bibitem[Pe2]{Pe2} {\sc V.V. Peller}, {\em Multiple operator integrals and higher operator
derivatives}, J. Funct. Anal.  {\bf233}  (2006),  515--544.

\bibitem[Pe3]{Pe3} {\sc V.V. Peller}, {\em Multiple operator integrals in perturbation theory}, Bull. Math. Sci. {\bf6} (2016), 15--88.

\bibitem[Pe4]{Pe4} {\sc V.V. Peller}, {\em Besov spaces in operator theory}, Russian Math. Surveys {\bf79:1}, 1--52. 

\bibitem[Pi]{Pi} {\sc J.D. Pincus}, {\em Commutators and systems of singular integral equations, I}, Acta Math. {\bf121} (1968), 219-249.

\end{thebibliography}
\end{document}